\documentclass[leqno,11pt,a4paper]{amsart}
\pdfoutput=1
\usepackage[full]{textcomp}
\usepackage{newpxtext}
\usepackage{comment}
\usepackage{cabin} 
\usepackage[varqu,varl]{inconsolata} 
\usepackage[bigdelims,vvarbb]{newpxmath} 
\usepackage[cal=cm,bb=esstix,bbscaled=1.05]{mathalfa} 
\usepackage[margin=2.395cm]{geometry}
\usepackage{savesym}
\let\savedegree\bigtimes
\let\bigtimes\relax
\usepackage{mathabx}
\usepackage{mathrsfs}
\usepackage{amsmath}
\usepackage{graphicx}
\let\bigtimes\savedegree
\usepackage[usenames,dvipsnames]{color}
\usepackage{hyperref}
\hypersetup{
 colorlinks,
 linkcolor={blue!90!black},
 citecolor={red!80!black},
 urlcolor={blue!50!black}
}
\usepackage{tikz}
\usepackage{tikz-cd}
\usepackage{cite}
\usepackage{enumitem}
\usepackage{amsfonts}

\setlist[enumerate]{labelsep=*, leftmargin=1.5pc}
\setlist[enumerate]{label=\normalfont(\roman*), ref=\roman*}
\newtheorem{thm}{Theorem}[section]
\newtheorem{lemma}[thm]{Lemma}

\newtheorem{prop}[thm]{Proposition}

\newtheorem{innercustomthm}{Theorem}
\newenvironment{customthm}[1]
  {\renewcommand\theinnercustomthm{#1}\innercustomthm}
  {\endinnercustomthm}

\theoremstyle{definition}

\newtheorem{notation}[thm]{Notation}

\newtheorem{remark}[thm]{Remark}
\newtheorem{definition}[thm]{Definition}
\newtheorem*{definition*}{Definition}
\newtheorem*{notation*}{Notation}
\newtheorem*{organization*}{Organization}
\newtheorem*{theorem*}{Theorem}
\numberwithin{equation}{section}

\newcommand{\kk}{\mathbf{k}}
\newcommand{\lex}{\text{lex}}
\newcommand{\inn}{\text{in}}

\newcommand{\rad}{\text{rad}}
\newcommand{\sat}{\text{sat}}
\newcommand{\GL}{\text{GL}}

\newcommand{\HH}{\mathcal{H}}

\newcommand{\PP}{\mathbf{P}}

\newcommand{\XX}{\mathcal{X}}
\newcommand{\YY}{\mathcal{Y}}

\DeclareMathOperator{\Hilb}{Hilb}
\DeclareMathOperator{\Proj}{Proj}
\DeclareMathOperator{\Spec}{Spec}

\makeatletter
\@namedef{subjclassname@2020}{%
  \textup{2020} Mathematics Subject Classification}
\makeatother
\begin{document}
\author[R.\,Ramkumar]{Ritvik~Ramkumar}
\address{Department of Mathematics\\University of California at Berkeley\\Berkeley, CA\\94720\\USA}
\email{ritvik@math.berkeley.edu}
\title{A Hilbert scheme of radius two}
\maketitle
\begin{abstract} We give an explicit example of a Hilbert scheme whose incidence graph has radius two.
\end{abstract}

\section{introduction}
Consider the Hilbert scheme $\Hilb^{P(t)} \PP^n$ that parameterizes subschemes of $\PP^n_{\kk}$ with a fixed Hilbert polynomial $P(t)$. A celebrated theorem of Reeves bounds the radius of the Hilbert scheme by $\deg P(t)+1$. More precisely, one can associate to $\Hilb^{P(t)} \PP^n$ its incidence graph: to each irreducible component we assign a vertex, and we connect two vertices if the corresponding components intersect. 
\begin{definition*}
Define the \textbf{distance} $d(C,D)$ between two components $C,D$ to be the number of edges in the shortest path linking the corresponding vertices. Let $r_D = \max \{ d(C,D) :C \text{ a component of }\linebreak \Hilb^{P(t)}  \PP^n\}$, and define the \textbf{radius} of the Hilbert scheme to be $\rad(\Hilb^{P(t)} \PP^n) =\min\{r_D : D \text{ a component of }  \Hilb^{P(t)} \PP^n\}$. We identify any component $D$ for which $\rad(\Hilb^{P(t)} \PP^n) = r_D$ as a \textbf{center} of the graph.
\end{definition*}

Every Hilbert scheme is connected \cite{Hartshorne} and contains a generically smooth irreducible component called the lexicographic component \cite{rs}. By studying this component in relation to other components Reeves established,
\begin{thm}[{\cite[Theorem 7]{reeves}}] Consider the Hilbert scheme $\Hilb^{P(t)} \PP^n$ and let $d = \deg P(t)$ be the dimension of the parameterized subschemes. Then the distance from any component to the lexicographic component is at most $d +1$. In particular, the radius of the Hilbert scheme is at most $d+1$.
\end{thm}
It is natural to ask to what extent Reeves' bound on the radius is sharp. As far as we are aware, no explicit example of a Hilbert scheme with radius larger than one has appeared in the literature. The goal of this note is to give an example of such a Hilbert scheme. 

Since the lexicographic component is, in general, the best understood component, one might start by studying the components which meet the lexicographic component. However, there are two immediate obstacles. The first is that it is difficult to determine all the components of the Hilbert scheme. Secondly, it is even more difficult to prove that two components of the Hilbert scheme do not meet. Even if we succeeded in determining which components meet the lexicographic component, the lexicographic component might not be the center of the incidence graph. We overcome these problems by working with an infinite family of Hilbert schemes where we completely understand a component different from the lexicographic component. For simplicity, we assume $\kk$ is an algebraically closed field of characteristic zero.

\begin{definition}
Let $n\geq 3$ and let $\HH^n= \Hilb^{P_n(t)} \, \PP^n$ be the Hilbert scheme corresponding to the Hilbert polynomial $P_n(t) = 2\binom{t+n-2}{n-2} - \binom{t+n-4}{n-4}$. Let $\mathcal{H}^n_1$ denote the irreducible component of $\HH^n$ whose general member parameterizes a pair of codimension two linear spaces meeting transversely in $\PP^n$. Let $\HH^n_2$ denote the component whose general member parameterizes $Q \cup \Lambda_{n-3}$ where $Q$ is a quadric $(n-2)$-fold and $\Lambda_{n-3}$ is a codimension three linear space such that $Q \cap \Lambda_{n-3}$ is a codimension four linear space. 
\end{definition}

\begin{thm}[{\cite[Theorem 1.1]{ccn}}\footnote{Our notation differs from \cite{ccn}; in their paper the authors use $H_n$ to denote the component $\HH^n_1$.}] \label{ccnthm} Let $n\geq 3$. The only component of $\HH^n$ that $\HH^n_1$ meets is $\HH^n_2$.
\end{thm}

Here is the main theorem of this note:

\begin{customthm}{A}
The radius of the Hilbert scheme $\HH^5$ is two. Moreover, the lexicographic component is not the center of the incidence graph.
\end{customthm}
With a bit more analysis, that we omit, we can describe a large portion of the incidence graph. In particular, other than the six known components  of $\HH^5$ \cite[Remark 2.7]{ccn} we found another component and we were able to determine how these components met one another. Moreover, we checked that all of these components are generically smooth. We believe that these are all the components, but we were unable to prove it:
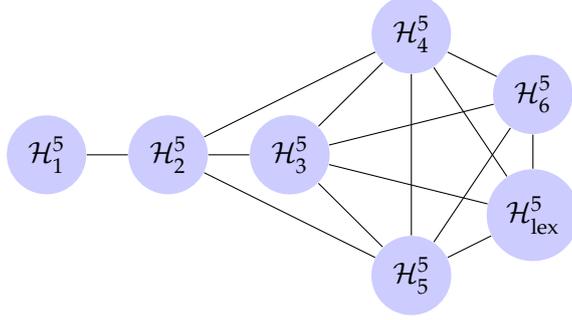
\begin{equation*}
\begin{tikzpicture}
  [scale=.8,auto=left,every node/.style={circle,fill=blue!20}]
  \node (n1) at (0,0) {$\HH_1^5$};
  \node (n2) at (2,0)  {$\HH_2^5$};
  \node (n3) at (4,0)  {$\HH_3^5$};
  \node (n4) at (6,2)  {$\HH_4^5$};
  \node (n5) at (6,-2) {$\HH_5^5$};
  \node (n6) at (8,1)  {$\HH_6^5$};
  \node (n7) at (8,-1)  {$\HH_{\lex}^5$};

  \foreach \from/\to in {n1/n2,n2/n3,n2/n4,n2/n5, n3/n4, n3/n5, n4/n5, 
  n4/n6, n4/n7, n5/n6, n5/n7, n6/n7, n3/n6, n3/n7}
    \draw (\from) -- (\to);

\end{tikzpicture}
\end{equation*}
Here is a description of the components appearing in the graph. For the rest of the paragraph, $\Lambda_i$ will denote an $i$-dimensional linear space and $Q$ will denote a quadric threefold. 
\begin{enumerate}
\item The general point of $\HH_3^5$ parameterizes the scheme theoretic union $Q \cup \Lambda_2 \cup Z$ where $Z$ is a double line of genus $-2$ embedded along $\Lambda_2$ and $Q \cap \Lambda_2$ is a conic.
\item The general point of $\HH_4^5$ parameterizes $Q \cup \Lambda_2 \cup \Lambda_1$ such that $Q$ and $\Lambda_2$ lie in a four dimensional linear subspace of $\PP^5$, and $Q\cap \Lambda_1$ is a point.
\item The general point of $\HH_5^5$ parameterizes $Q \cup \Lambda_2 \cup \Lambda_1$ such that $Q$ and $\Lambda_2$ lie in a four dimensional linear subspace of $\PP^5$, and $\Lambda_2 \cap \Lambda_1$ is a point
\item The general point of $\HH_6^5$ parameterizes $Q \cup \Lambda_2 \cup \Lambda_1 \cup \Lambda_0$ such that $Q$, $\Lambda_2$ and $\Lambda_1$ lie in a four dimensional linear subspace of $\PP^5$, and $\Lambda_0$ is an isolated point. 
\item The general point of $\HH_{\lex}^5$ parameterizes  $Q \cup \Lambda_2 \cup \Lambda_1 \cup \Lambda_0 \cup \Lambda_0'$ such that $Q$, $\Lambda_2$ and $\Lambda_1$ lie in a four dimensional linear subspace of $\PP^5$, $\Lambda_1 \cap \Lambda_2$ is a point, and $\Lambda_0,\Lambda_0'$ are isolated points.
\end{enumerate}
 
\smallskip
It is conceivable that $\lim_{n \to \infty} \rad(\HH^n) = \infty$; however, investigating this is beyond our current capabilities. Thus we make the following conjecture,

\newtheorem*{conj:main'}{Conjecture B}
\begin{conj:main'} There exists a family of Hilbert schemes $\{\Hilb^{Q_d(t)} \PP^{n_d} \}_{d \in \mathbf{N}}$ such that $\deg Q_d(t) = d$ and  $\rad(\Hilb^{Q_d(t)} \PP^{n_d}) = O(d)$ as $d \to \infty$.
\end{conj:main'}

\section{Computing the radius}
We begin the section by fixing some notation and terminology. Throughout the paper $S$ will denote the polynomial ring $\kk[x_0,\dots,x_n]$. Given an ideal $I \subseteq S$, by abuse of notation, we use $I$ or $Z$ to denote the $\kk$-point in the Hilbert scheme corresponding to $Z = \Proj(S/I) \subseteq \PP^n$. The ideal associated to a subscheme always refers to its saturated ideal. By a component of $\HH^n$ we always mean an irreducible component.  For facts about the lexicographic component, including a description of its general point we refer to \cite{rs}. 

The group $\GL_{n+1}$ acts on $S$ and $\HH^n$ by a change of coordinates. An ideal of $S$ or its corresponding point on the Hilbert scheme is said to be \textbf{Borel-fixed} if it is fixed by the Borel subgroup of $\GL_{n+1}$ consisting of upper triangular matrices. Since a Borel-fixed ideal is fixed by the subgroup of diagonal matrices, it is generated by monomials. Borel's fixed point theorem implies that for any $I$ in $\HH^n$ there is a one-parameter family whose general fiber is $I$ and whose special fiber is a Borel-fixed ideal. Moreover, every component in $\HH^n$ and every intersection of components of $\HH^n$ contains a Borel-fixed ideal. For more details on their structure we refer to \cite[Chapter 15]{Eisenbud}.

\smallskip
Prior to analyzing $\HH^5$ we need a sufficiently good understanding of $\HH^4$. The general point of $\HH^4_{\lex}$ parameterizes a quadric surface union a line and two isolated points, such that the line meets the quadric at two points.

\begin{lemma} \label{BorelG} The Hilbert scheme $\HH^4$ has three Borel-fixed ideals:
$$
I_1= (x_0^2,x_0x_1,x_0x_2,x_1^2), \quad
I_2 = (x_0^2,x_0x_1,x_0x_2,x_0x_3,x_1^3,x_1^2x_2), \quad
I_{\text{lex}} = (x_0,x_1^3,x_1^2x_2^2,x_1^2x_2x_3). 
$$
Moreover,
\begin{enumerate}
\item $I_1$ only lies in  $\HH_1^4$ and $\HH^4_2$,
\item $I_{\lex}$ only lies in $\HH^4_{\lex}$, 
\item $I_2$ is in every component of $\HH^4 \setminus \HH^4_1$. 
\end{enumerate}
\end{lemma}
\begin{proof} The Borel-fixed ideals can be computed using \cite[Algorithm 4.6]{moo} or using the computer algebra system \textit{Macaulay2} \cite{m2} and the package \textit{Strongly stable ideals} \cite{al}. By \cite[Theorem C]{ritvik}\footnote{Once again, our notation differs with \cite{ritvik}; in that paper $\HH^n_{n-2,n-2}$ is used to denote the component $\HH^n_1$.} $I_1$ is the unique Borel-fixed ideal on $\HH^4_1$. Since $\HH^4_1$ meets $\HH^4_2$ and their intersection must contain a Borel-fixed ideal, $I_1$ also lies in $\HH^4_2$. Since $\HH^4_1$ does not meet any other component (Theorem \ref{ccnthm}), $I_1$ does not lie on any other component. It is well known that the lexicographic ideal, $I_{\lex}$, is a smooth point \cite{rs} and thus it lies on its own component, $\HH_{\lex}^4$. Since $\HH^4$ is connected, every component of $\HH^4 \setminus \HH_1^4$ contains $I_2$.
\end{proof}

\begin{prop} \label{HilbG} The Hilbert scheme $\HH^4$ has radius one while the distance between $\HH^4_1$ and $\HH^4_{\lex}$ is two.
\end{prop}
\begin{proof} This is an immediate consequence of Lemma \ref{BorelG} as every component of $\HH^4$ meets $\HH^4_2$ and $\HH^4_{\lex}$ does not meet $\HH^4_1$.
\end{proof}
This shows that even when the radius is one, the lexicographic component need not be the center of the incidence graph. 

\begin{remark} \label{classify} By computing a neighbourhood of $I_2$ in $\HH^4$, it can be shown that $\HH^4_1,\HH^4_2,\HH^4_{\lex}$ are the only irreducible components of $\HH^4$ and that $\HH^4_2$ is smooth. 
\end{remark}

\begin{lemma} \label{BorelH} The Hilbert scheme $\HH^5$ has nine Borel-fix ideals:
\begin{enumerate}\setlength\itemsep{0.5em}
\item $I_1 = I_{\lex} =  (x_0,x_1^3,x_1^2x_2^2,x_1^2x_2x_3^2,x_1^2x_2x_3x_4^2)$,
\item $I_2 = (x_0,x_1^3,x_1^2x_2x_3x_4,x_1^2x_2^2x_4,x_1^2x_2x_3^2,x_1^2x_2^2x_3,x_1^2x_2^3)$,
\item $I_3 = (x_0,x_1^4,x_1^3x_2,x_1^3x_3,x_1^3x_4,x_1^2x_2^2,x_1^2x_2x_3^2,x_1^2x_2x_3x_4)$,
\item $I_4 = (x_0,x_1^4,x_1^3x_2,x_1^3x_3,x_1^2x_2^2,x_1^2x_2x_3,x_1^3x_4^2)$,
\item $I_5 = (x_0^2,x_0x_1,x_0x_2,x_0x_3,x_0x_4,x_1^3,x_1^2x_2x_3^2,x_1^2x_2x_3x_4,x_1^2x_2^2)$,
\item $I_6 = (x_0^2,x_0x_1,x_0x_2,x_0x_3,x_0x_4,x_1^4,x_1^3x_2,x_1^3x_3,x_1^3x_4,x_1^2x_2^2,x_1^2x_2x_3)$,
\item $I_7 = (x_0^2,x_0x_1,x_0x_2,x_0x_3,x_0x_4^2,x_1^3,x_1^2x_2x_3,x_1^2x_2^2)$,
\item $I_8 = (x_0^2,x_0x_1,x_0x_2,x_0x_3,x_1^3,x_1^2x_2)$,
\item $I_9 = (x_0^2,x_0x_1,x_0x_2,x_1^2)$.
\end{enumerate}
Moreover, $I_1,\dots,I_7$ are the only Borel-fixed ideals lying in the lexicographic component.
\end{lemma}
\begin{proof} The computation of Borel-fixed ideals is similar to Lemma \ref{BorelG}. To prove the other statement we appeal to a theorem of Reeves. Given an ideal $J \subseteq S$ we define the \textit{double saturation}, $\sat_{x_{4},x_5}(J)$ to be the ideal obtained by setting $x_4=1$ and $x_5 = 1$ in $J$.
It is shown in \cite[Theorem 11]{reeves} that a Borel-fixed ideal $J$ lies in the lexicographic component if and only if $\sat_{x_4,x_5}(J) = \sat_{x_4,x_5}(I_{\lex})$. It is clear that the double saturation of $I_1,\dots,I_7$ are all equal to $\sat_{x_4,x_5}(I_{\lex}) = (x_0,x_1^3,x_1^2x_2^2,x_1^2x_2x_3)$ while the double saturation of $I_8$ and $I_9$ are different.
\end{proof}

\begin{notation} Let $Z_j$ denote the Borel-fixed points defined by the ideal $I_j$ of Lemma \ref{BorelH}.
\end{notation}

\begin{lemma} \label{H2} The component $\HH^5_2$ does not meet $\HH_{\lex}^5$. Moreover, the only Borel-fixed points on $\HH_5^2$ are $Z_8$ and $Z_9$. 
\end{lemma}
\begin{proof} By Lemma \ref{BorelH} it suffices to show that $\HH^5_2$ does not contain $Z_1,\dots,Z_7$. Assume this was not the case; then there is a flat family $\XX \to \Spec \kk[t]_{(t)}$ with generic fiber $\XX_{\{(0)\}}$ isomorphic to a quadric threefold meeting a plane along a line and special fiber $\XX_{\{(t)\}} = Z_i$ for some $i \leq 7$. We may choose the family so that $\XX_{\{(0)\}}$ is transverse to the hyperplane $V(x_5)$ in $\PP^{5}_{\kk(t)}$. Since $x_5$ is a non-zero divisor on $S/I_{Z_i}$, the hyperplane section $ \XX \cap V(x_5) \to \Spec \kk[t]_{(t)}$ is still flat.

Since $\XX_{\{(0)\}} \cap V(x_5)$ is a quadric surface meeting a line at a point, $Z_i \cap V(x_5)$ must lie in the component $\HH^4_2$. A straightforward computation shows that the (saturated) ideal of $Z_i \cap V(x_5)$ is defined by 
$
(x_5,x_0,x_1^3,x_1^2x_2x_3,x_1^2x_2^2).
$
But as noted in Lemma \ref{BorelG} (ii), this defines the lexicographic point which lies in $\HH^4_{\lex} \setminus \HH^4_2$; a contradiction.

By \cite[Theorem C]{ritvik}, $Z_9$ is the unique Borel-fixed point in $ \HH^5_1$ and thus $Z_9 \in \HH^5_1 \cap \HH^5_2 \subseteq \HH^5_2$. Since the Hilbert scheme is connected, $\HH^5_2$ must meet a component $\YY$ different from $\HH_1^5$ and $\HH_{\lex}^5$. Once again using Lemma \ref{BorelH} we see that $Z_8 \in \HH^5_2 \cap \YY \subseteq \HH^5_2$.
\end{proof}

\begin{proof}[Proof of Theorem A] Since $\HH_1^5$ only meets $\HH_2^5$ (Theorem \ref{ccnthm}) and $\HH_{\text{lex}}^5$ does not meet $\HH_2^5$ (Lemma \ref{H2}), the radius of $\HH^5$ is at least two. To show that the radius of $\HH^5$ is at most two, it is enough to establish the following two facts:
\begin{enumerate}
\item The distance from $\HH_{2}^5$ to $\HH_{\lex}^5$ is two,
\item If $\mathcal{Y}$ is a component of $\HH^5$ that does not meet $\HH^5_2$ then $\YY$ meets $\HH_{\lex}^5$.
\end{enumerate}
Indeed, once we know these two facts, the component connecting $\HH_2^5$ to $\HH_{\lex}^5$ will be a center of the incidence graph. To prove (i) consider a path $\HH_{2}^5 = \YY_1,\YY_2,\dots,\YY_m = \HH_{\lex}^5$ with $\YY_i \cap \YY_{i+1} \ne \emptyset$ and $m$ minimal. The minimality of $m$ implies $\YY_3 \cap \YY_1 = \emptyset$. Since $Z_8,Z_9$ lie in $\YY_1$, the intersection $\YY_2 \cap \YY_3$ must contain one of $Z_1,\dots,Z_7$. By Lemma \ref{BorelH}, $\YY_2$ meets the lexicographic component. Thus $m =3$ proving item (i). The proof of item (ii) is analogous.
\end{proof}

\end{document}